\documentclass[letterpaper,12pt,reqno]{amsart}

\usepackage{amsmath,fancyhdr}
\usepackage{amssymb}
\usepackage{amsfonts}
\usepackage{amscd}
\usepackage{amsthm}
\usepackage{graphicx}
\usepackage{hyperref}
\usepackage{tikz,tikz-3dplot}
\usepackage[ansinew]{inputenc}

\begin{document}

\title{Tropical Incidence Relations, Polytopes, and Concordant Matroids}
\author{Mohammad Moinul Haque}
\maketitle

\newtheorem{defn}{Definition}
\newtheorem{lem}{Lemma}
\newtheorem{propn}{Proposition}
\newtheorem{thm}{Theorem}
\newtheorem{cor}{Corollary}
\newtheorem{conj}{Conjecture}
\newtheorem{ex}{Example}
\newcommand{\sub}{\backslash}
\newcommand{\B}{\mathcal{B}}

\theoremstyle{remark}
\newtheorem{rem}{Remark}

\begin{abstract}
In this paper, we develop a tropical analog of the classical flag variety that we call the flag Dressian.  We find relations, which we call ``tropical incidence relations'', for when one tropical linear space is contained in another, and show that the flag Dressian is a tropical prevariety.  In the case of 2-step flag Dressians, which we call ``tropical incidence prevarieties'', we find an equivalence between points in this space and induced subdivisions of a hypersimplex, generalizing two parts of an equivalence given by D. Speyer for tropical linear spaces (Proposition 2.2 of \cite{Speyer}).  We attempt to generalize the third part of Speyer's equivalence to concordant matroids and obtain some partial results.
\end{abstract}

\section{Introduction}

\subsection{Background}

Tropical geometry concerns the study of tropical spaces, which are certain weighted polyhedral complexes in $\mathbb{R}^{n}$.  Tropical spaces are analogs of algebraic varieties that are useful for being able to carry information about varieties despite being geometrically simpler.  In fact, there is a method for directly converting algebraic spaces into tropical ones called ``tropicalization.'' In addition, tropical spaces have interesting and useful combinatorial properties. One example of such a space is a tropical linear space, a tropical analog of a classical linear space.

In \cite{SS} tropical linear spaces are introduced in the context of $Trop(G(d,n))$, the tropical Grassmannian, the tropical analog of the classical Grassmannian $G(d,n)$.  However, the tropical Grassmannian only parametrizes \textit{realizable} tropical linear spaces, namely, those that can be obtained by tropicalizing classical linear spaces.  It is the \textit{Dressian}, denoted $Dr(d,n)$, which parametrizes all $d$-dimensional tropical linear spaces in $n$-dimensional tropical projective space, and is introduced in (\cite{TropPlane}).  While $Trop(G(2,n))\cong Dr(2,n)$, the Dressian is in general larger than the Grassmannian for the same values of $d$ and $n$ (for a given $d$, $d>2$, the dimension of the Dressian increases quadratically in $n$, but only linearly for the Grassmannian, [Theorem 3.6 of \cite{TropPlane}]).

In (\cite{Speyer}) Speyer discusses tropical linear spaces and introduces tropical Plücker vectors.  Given a collection of real numbers $\{p_{J}; J\subset[n],\;|J|=d\}$ satisfying:
\[
(p_{J})\in Trop(P_{S_{ij}}P_{S_{kl}}-P_{S_{ik}}P_{S_{jl}}+P_{S_{il}}P_{S_{jk}})
\]
for every $S\in{[n]\choose{d-2}}$ and $i,j,k$ and $l$ distinct elements of $[n]\backslash S$, we call such a set of numbers a \textbf{tropical Plücker vector}, and the numbers themselves \textbf{tropical Plücker coordinates}.  To each vertex $e_{i_{1}}+\cdots+e_{i_{d}}$ of $\Delta(d,n)$ we associate the number $p_{i_{1},\ldots,i_{d}}$, which we use to take the lower subdivision called $\mathcal{D}$(Section 2 of \cite{Speyer}).  He has a particular characterization of linear spaces which he gives in a proposition (Proposition 2.2 of \cite{Speyer}):

\begin{propn}
\label{Sp}
The following are equivalent:

1. $p_{i_{1},\ldots,i_{d}}$ are tropical Plücker coordinates

2. The one skeleta of $\mathcal{D}$ and $\Delta(d,n)$ are the same.

3. Every face of $\mathcal{D}$ is matroidal.
\end{propn}

Here the Plücker coordinates are those of the tropical linear space, and $\mathcal{D}$ is the subdivision of $\Delta(d,n)$ induced by the Plücker coordinates.  He goes on to use matroid theory to characterize tropical linear spaces and uses the above proposition to determine properties of these spaces, such as the fact that stable intersections of tropical linear spaces is also a tropical linear space.

\subsection{Goals and Results}

The goals of this paper are threefold: The first is to introduce the flag Dressian, a tropical space parametrizing flags of tropical linear spaces.  Next, we show an equivalence between inclusion of tropical linear spaces and subdivisions of weight polytopes related to matroid polytopes.  Finally, we make a connection from tropical incidence to matroid concordance.

The first section is devoted to the flag Dressian, a moduli space of flags of tropical linear spaces, and thus a tropical analog of flag varieties.  This section provides the tropical incidence relations (Theorem 1 of section 2.3.2), which show that the flag Dressian is an intersection of tropical hyperplanes.  The remainder of the paper focuses on the case of flag Dressians where there is only one step in the flag, and which are termed ``tropical incidence prevarieties.''  This class of varieties alone provides us with some interesting results.

The following section introduces the weight polytope $\Delta$, which are related to tropical incidence varieties.  In particular, flags of tropical linear spaces induce subdivisions on $\Delta$.  This section characterizes these induced subdivisions uniquely as those that share the same
1-skeleton with $\Delta$ (Theorem 2 of section 3.1).  This theorem is an analog of statements 1 and 2 of Speyer's Proposition \ref{Sp} above for tropical incidence varieties.

The last section focuses on determining a relationship between the polynomial and polyhedral conditions in the first two sections to matroids.  While a complete analogy appears to be elusive, we obtain some partial results and a conjecture for what may be possible.

\section{Relations for a Flag of Tropical Linear Spaces}

Here we derive a set of relations for tropical flag prevarieties using duality.  We then compare the equations we get to the ones that generate classical flag varieties.  Our approach will be to first define a ``tropical duality" that converts Plücker coordinates of a space into Plücker coordinates of its orthogonal space.  We then use the definition of a tropical linear space given by Speyer and apply duality to derive relations for when one tropical linear space is contained in another, giving us relations for a tropical flag variety.

\subsection{Tropical Flag Prevarieties}

Before we can discuss tropical objects, we first have to discuss what it means for something to be ``tropical".  Tropical geometry essentially converts classical algebraic varieties into piecewise linear spaces in order to get at certain aspects of the geometry.  The method of \textbf{tropicalizing} works as follows: given a polynomial $f$ in $n$ variables over a field with valuation, say
\[
f=\sum a_{r}\textbf{X}^{\textbf{w}_{r}},
\]
where $\textbf{X}^{\textbf{w}_{r}}=X_{1}^{{w_{r}}_{1}}X_{2}^{{w_{r}}_{2}}\cdots X_{n}^{{w_{r}}_{n}}$ then we define the \textbf{tropical zero set}, denoted $Trop(f)$, to be the set of points in $\mathbb{R}^{n}$ such that the minimum of $\{val(a_{r})+\sum_{i=1}^{n} {w_{r}}_{i}\}$ is achieved at least twice, following the convention of \cite{Speyer} (depending on conventions, some use maximum).  There are other equivalent ways of defining tropicalization, such as taking valuations of polynomials defined over a function field.  See \cite{firststep} for a more complete discussion of tropical geometry, and \cite{Gat} for a survey of tropical geometry and the different types of tropical spaces.\\
\\
\emph{Convention}: We will use uppercase letters, such as $X$ to represent variables before tropicalizing, and lowercase letters, such as $x$ to represent tropicalized variables.\\
\\
We want to develop a notion of a moduli space representing a flag of tropical linear spaces.  First, let us make clear what we mean by a ``tropical linear space".  We use the definition of tropical linear space as given in \cite{Speyer} (following Proposition 2.2) as follows:

\begin{defn} 
\label{tropl}
Let $\{p_{J}; J\subset[n],\;|J|=d\}$ be a collection of real numbers satisfying the tropical plucker relations:
\[
(p_{J})\in Trop(P_{S_{ij}}P_{S_{kl}}-P_{S_{ik}}P_{S_{jl}}+P_{S_{il}}P_{S_{jk}})
\]
where $S\in{[n]\choose{d-2}}$ and $i,j,k$ and $l$ are distinct elements of $[n]\backslash S$.  We call such a set of numbers a \textbf{tropical Plücker vector}, and the numbers themselves \textbf{tropical Plücker coordinates}.  The \textbf{$(d-1)$-dimensional tropical linear space} with Plücker coordinates $P_{J}$ is the subset of $\mathbb{R}^{n}/(1,\ldots,1)$ given by
\begin{equation}
\label{troplinear}
\bigcap_{1\leq j_{1}<\cdots<j_{d+1}\leq n}Trop\left(\sum_{r=1}^{d+1}(-1)^{r}P_{j_{1}\cdots\hat{j_{r}}\cdots j_{d+1}}X_{j_{r}}\right).
\end{equation}
\end{defn}

Now that we have defined tropical linear spaces, we can then consider flags of tropical linear spaces.  As the name \textit{Dressian} has been suggested for the tropical prevariety parametrizing all tropical linear spaces of a given dimension (See \cite{TropPlane}), we propose the name \textit{flag Dressian} to differentiate it from the tropicalization of flag varieties.

\begin{defn}: Let the \textbf{flag Dressian}, denoted, $FD(d_{1},d_{2},\ldots,d_{s};n)$ be the moduli space of flags of tropical linear spaces in $\mathbb{R}^{n}/(1,\ldots,1)$; namely, the set of $s$-tuples $(X_{1},\ldots,X_{s}))$ such that each $X_{i}$ is a $(d_{i}-1)$-dimensional tropical linear space in $\mathbb{R}^{n}/(1,\ldots,1))$ and $X_{i}\subset X_{i+1}$ for every $i$.  We will show below that the flag Dressian is in fact a tropical variety.  In the case where $s=2$ we will also call the flag Dressian a \textbf{tropical incidence variety}.
\end{defn}

Note, however, that $FD(d_{1},d_{2},\ldots,d_{s};n)$ is not the tropicalization of the classical flag variety, $Trop(Fl(d_{1},d_{2},\ldots,d_{s};n))$.  In particular, an element $(L_{1},\ldots,L_{s})\in FD(d_{1},d_{2},\ldots,d_{s};n)$ is a flag of tropical linear spaces where none of the $L_{i}$ have to be realizable, but for $(K_{1},\ldots,K_{s})\in Trop(Fl(d_{1},d_{2},\ldots,d_{s};n))$ all of the $K_{i}$ must be realizable.  One may ask if it is possible there is $(X_{1},X_{2})\in FD(d_{1},d_{2};n)$ where each $X_{i}$ is realizable, but for any $A,B$ with $Trop(A)=X_{1}$ and $Trop(B)=X_{2}$ we have $(A,B)\not\in Fl(d_{1},d_{2};n)$, namely $A\not\subset B$ (so containment may not be realizable).

\subsection{Relations for Tropical Incidence Prevarieties}

Having defined the flag Dressian, in order to determine its structure, we need a condition for when one tropical linear space is contained in another.  In this case, we will develop relations on the Plücker coordinates that determine containment.  One of our main tools for determining these relations will be duality. 

\subsubsection{Duality}

\begin{lem}:
\label{pdual}
Let $\{x_{I}\}$ be tropical Plücker coordinates for a tropical linear space $X$ and define $x^{\perp}_{J}:=x_{[n]\backslash J}$.  Then $\{x^{\perp}_{J}\}$ are also tropical Plücker coordinates for a tropical linear space we denote by $X^{\perp}$.  We call $X^{\perp}$ the tropical \textbf{dual linear space}, or \textbf{orthogonal complement} to $X$.  Moreover, $X\mapsto X^{\perp}$ is an inclusion reversing bijection on tropical linear spaces.
\end{lem}

\begin{proof}: See the remarks preceding Proposition 2.12 in \cite{Speyer}.
\end{proof}

\subsubsection{Tropical incidence relations}

We will now make use of (\ref{troplinear}) and duality to derive relations for a tropical incidence variety.  We will show that in fact, the relations that we derive for tropical incidence prevarieties are the same as the relations for classical incidence varieties.

\begin{thm}: 
\label{tropincl}
Given tropical linear spaces $X$ and $Y$ ($dim\;X\leq dim\;Y$), we have $X\subset Y$ if and only if the tropical Plücker coordinates of $X$ and $Y$ satisfy the incidence relations:
\[
Trop(\sum_{i\in T\backslash S}X_{S\cup\{i\}}Y_{T\backslash\{i\}}),
\]
where $S,T\subset[n]$ with $|S|=p-1$ and $|T|=q+1$.
\end{thm}
\begin{proof}:  Let $q=dim\;Y$, then from (\ref{troplinear}) we know that $Y$ is the intersction over all sequences $J$ of the tropical hyperplane $H_{J}$ given by
\[
Trop\left(\sum_{j\in J}Y_{J\backslash{j}}W_{j}\right),
\]
(where the $W_{j}$ represent the variables).  Note that we can ignore the signs in front of the terms since we are tropicalizing.  We have $X\subset Y$ if and only if $X\subset H_{J}\;\forall J$ if and only if $H^{\perp}_{J}\subset X^\perp\;\forall J$.\\
\\
Let $p=dim\;X$, then we also have $X^{\perp}$ as the intersection of tropical hyperplanes $H_{I'}$ of the form
\[
Trop\left(\sum_{r=1}^{n-p+1}X^{\perp}_{i'_{1}\cdots\hat{i'_{r}}\cdots i'_{n-p+1}}W_{i'_{r}}\right).
\]
Let $H^{\perp}_{J}=Trop(V)$.  Using Lemma \ref{pdual} we can write the coordinates of $V$ as
\[
V_{j}=
\begin{cases}
Y_{J\backslash j} & \text{for} j\in J\\
0 & \text{for} j\not\in J.
\end{cases}
\]
Now $H^{\perp}_{J}\subset X^\perp\;\forall J$ if and only if $\forall J,I\: H^{\perp}_{J}\subset H_{I'}$, so plugging in the points $H^{\perp}_{J}$ into the equations for $H_{I'}$ gives
\begin{equation}
Trop\left(\sum_{r=1}^{n-p+1}X^{\perp}_{i'_{1}\cdots\hat{i'_{r}}\cdots i'_{n-p+1}}Y_{J_{i'_{r}}}\right)
\end{equation}
which, by Lemma \ref{pdual}, is equivalent to
\begin{equation}
Trop\left(\sum_{r=1}^{n-p+1}X_{[n]\backslash\{i'_{1}\cdots\hat{i'_{r}}\cdots i'_{n-p+1}\}}Y_{J_{i'_{r}}}\right).
\end{equation}
We write the above equation more conveniently as
\begin{equation}
\label{tropinc}
Trop\left(\sum_{r=1}^{n-p+1}X_{\{[n]\backslash I'\}\cup\{i'_{r}\}}Y_{J_{i'_{r}}}\right).
\end{equation}
Taken over all possible sequences, tropicalizing and then intersecting the tropicalizations gives us a tropical analog to the classical incidence variety.\\
\\
Note that we can simplify equation (\ref{tropinc}) by using $S:=[n]\backslash I'$ and $T:=J$ and writing the above equation as
\begin{equation}
\label{tsiminc}
Trop(\sum_{i\in T\backslash S}X_{S\cup\{i\}}Y_{T\backslash\{i\}}).
\end{equation}
These equations are the tropicalization of the classical incidence relations:
\begin{equation}
\label{siminc}
\sum_{i\in T\backslash S}X_{S\cup\{i\}}Y_{T\backslash\{i\}}
\end{equation}
\end{proof}

\section{Polytopes Associated to Incidence Varieties}

We will describe the structure of $FD(p,q;n)$ using the combinatorics of the polytope $Conv(\Delta(p,n)\cup\Delta(q,n))$ where 
\[
\Delta(d,n):=Conv(\{e_{i_{1}}+e_{i_{2}}+\cdots+e_{i_{d}};\{i_{1},\ldots,i_{d}\}\in\binom{[n]}{k}\}).
\]
For simplicity, we define 
\[
\Delta(p,q;n):=Conv(\Delta(p,n),\Delta(q,n)),
\]
which we will simply write as $\Delta$, and 
\[
e_{i_{1},\ldots,i_{d}}:=e_{i_{1}}+\cdots+e_{i_{d}}.
\]
The relationship between the variety and the polytope is as follows:  A point of $FD(p,q;n)$ is a pair of tropical linear spaces $V$ and $W$ such that $V\subset W$.  Let $V$ and $W$ have tropical Plücker vectors given by $x_{I}$ and $y_{J}$ respectively (as defined in Definition \ref{tropl}).  This lets us define a function from vertices, $e_{i_{1},\ldots,i_{d}}$, of $\Delta$ to plucker vectors, $p_{i_{1},\ldots,i_{d}}$, by mapping 
\[
e_{I}\mapsto x_{I}
\]
and 
\[
e_{J}\mapsto y_{J}.
\]
The lower convex hull of the points $(e_{i_{1},\ldots,i_{d}},p_{i_{1},\ldots,i_{d}})\in \Delta\times\mathbb{R}$ project to give a subdivision of $\Delta$.  (See also Section 2 of \cite{Speyer} and Chapter 7 of \cite{GKZ}.)  We will denote by $\mathcal{D}$ the \textbf{subdivision of $\Delta$ arising from the tropical Plücker coordinates} $\{x_{I},y_{J}\}$.

\emph{Convention}: Throughout this section we will take $p+q\leq n$, since if $p+q\geq n$ then we can simply dualize the variety to get $(n-q)+(n-p)\leq n$.  Since the number of equations and terms in each are the same for the variety and its dual, it suffices to only consider the cases where $p+q\leq n$.

\subsection{Subdivisions of the polytope}

We would like to characterize the possible subdivisions $\mathcal{D}$ of $\Delta$ arising from points $\{x_{I},y_{J}\}$ that are tropical Plücker coordinates satisfying the incidence relations (\ref{tsiminc}).

\begin{defn}: For $n\in\mathbb{N}$ we define the polytope $P_{n}:=Conv(\{\pm e_{i};i\in[n]\})$.
\end{defn}

Given a function 
\[
\psi:\{\pm e_{i};i\in[n]\}\rightarrow\mathbb{R}
\]
on the vertices of $P_{n}$, define $\psi_{i}^{+}:=\psi(e_{i})$ and $\psi_{i}^{-}:=\psi(-e_{i})$ and let $w_{i}:=\psi_{i}^{+}+\psi_{i}^{-}$.  We then have the following:

\begin{propn}: Assume, without loss of generality that $w_{i}\leq w_{i+1}\,\forall i$.  Then $w_{1}<w_{2}$ if and only if the subdivision induced by $\psi$ includes the edge $e_{1},-e_{1}$.  Moreover, if $w_{1}=w_{2}$ then the subdivision introduces no new edges.
\label{min}
\end{propn}

\begin{proof}:
Since the graph of $P_{n}$ under $\psi$ sits in $\mathbb{R}^{n+1}$ we can take $n+1$ vertices of the graph to form a hyperplane.  We can then define a linear function $L(x)$ on $\mathbb{R}^{n}$ such that the graph of $L(x)$ is the hyperplane.  In particular, choose those vertices to be $(-e_{i},\psi(-e_{i}))$, $\forall i\in[n]$, and $(e_{1},\psi(e_{1}))$ and $H$ the hyperplane through those points.  Choose a linear functional $L(x)$ so that $(x,L(x))\in H$, and define
\[
\psi'(x):=\psi(x)-L(x).
\]
Note that the subdivision induced by $\psi'$ is the same as that induced by $\psi$.  We have $\psi'(-e_{i})=0$, $\psi'(e_{1})=0$ and $\psi'(e_{i})\geq0$ with $\psi'(e_{i})=0$ if and only if $w_{i}=w_{1}$, which implies that $Conv(\{e_{j}\},\{-e_{i}\};\,i\in[n],\psi'(e_{j})=0)$ is a face of the subdivision of $P_{n}$ induced by $\psi'$, and therefore $\psi$.  More specifically, we have that $e_{1},-e_{1}$ is an edge of the induced subdivision if and only if $w_{1}<w_{2}$.
\end{proof}

Having described the polytopes $P_{n}$ let us relate them to the polytope $\Delta$.  Every incidence relation written in the form (\ref{tsiminc}) is determined by multi-index sets $S$ and $T$ of sizes $p-1$ and $q+1$ respectively.  To each such equation we associate the following polytope:
\[
\Delta_{S,T}:=Conv({e_{S\cup\{i\}},e_{T\backslash\{i\}};i\in T\backslash S}).
\]

\begin{propn}: $\Delta_{S,T}$ is a face of $\Delta$.
\label{face}
\end{propn}

\begin{proof}: Consider the linear function $L(x)=\sum_{i\in S}x_{i}-\sum_{j\notin T}x_{j}$.  This function satisfies
\[
L(e_{j})=
\begin{cases}
 1, &j\in T\cap S\\
 -1, &j\not\in T\cup S\\
 0, &else.
\end{cases}
\]
Hence we always have that $L(e_{I})\leq|T\cap S|$ for any multi-index set $I$, and in particular $L(e_{I})=|T\cap S|$ if and only if 
\begin{equation}
T\cap S\subset I\subset T\cup S.
\label{lineq}
\end{equation}
This means that for the vertices of $\Delta_{S,T}$ we have $L(e_{S\cup\{i\}})=L(e_{T\backslash\{j\}})=|T\cap S|,\;\forall i,j\in T\backslash S$.  Any other vertex of $\Delta_{S,T}$ is either of the form $e_{S'\cup\{i'\}}$ where $|S'|=p-1$ and $S'\neq S$ or $e_{T'\backslash\{j'\}}$ where $|T'|=q+1$, $T'\neq T$.  Note that if we have  $T\cap S\subset S'\cup\{i'\}\subset T\cup S$, then in fact $e_{S'\cup\{i'\}}$ is a vertex of $\Delta_{S,T}$ since we can express $S'\cup\{i'\}$ in the form $S\cup\{i\}$ for some $i\in T\backslash S$.  Similarly for $T'\backslash\{j'\}$.  Hence, if $I=S'\cup\{i'\}$ or $I=T'\backslash\{j'\}$ are multi-index sets that do not satisfy (\ref{lineq}) then $L(e_{I})<|T\cap S|$.  As only the vertices of $\Delta_{S.T}$ are the only vertices of $\Delta$ that take on the value $|T\cap S|$ under the linear functional $L$, and all other vertices of $\Delta$ take on values strictly less than $|T\cap S|$, the proposition follows.
\end{proof}

We now claim that the two types of polytopes, $P_{n}$ and $\Delta_{S,T}$ are equivalent for the purposes of what we are trying to do:

\begin{propn}: There is an affine transformation taking $\Delta_{S,T}$ to $P_{|T\backslash S|}$.
\label{trans}
\end{propn}

\begin{proof}: We will start by showing a transformation exists for $S=\emptyset$, then showing that any $\Delta_{S,T}$ can be transformed to $\Delta_{\emptyset,T'}$ for some $T'$.\\
Without loss of generality, let us take $T'=[n]$.  Given $i\in[n]$, the vertices $e_{i}$ and $e_{[n]\backslash \{i\}}$ of $\Delta_{\emptyset,T'}$ are opposite each other.

Consider the affine transformation
\[
x\mapsto Ax+b
\]
given by
\[
A(e_{i}):=e_{i}-\frac{1}{n-2}\sum_{j\in[n]}e_{j}
\]
and
\[
b:=\frac{1}{n-2}\sum_{j\in[n]}e_{j}.
\]
One can easily verify that this transformation satisfies 
\begin{equation}
e_{i}=A(e_{i})+b
\end{equation}
and
\begin{equation}
-e_{i}=A(e_{[n]\backslash\{i\}})+b.
\label{2}
\end{equation}
Now we only need to show that there is a affine transformation from $\Delta_{S,T}$ to $\Delta_{\emptyset,T'}$.  This transformation is given by sending $S$ to $S':=\emptyset$ and $T$ to $T':=T\backslash S$.  On vertices this sends $e_{S\cup\{i\}}$ to $e_{i}$ and $e_{T\backslash\{i\}}$ to $e_{\{T\backslash S\}\backslash\{i\}}=e_{T'\backslash\{i\}}$.  This is given by the affine transformation
\[
s\mapsto A'x+b
\]
where 
\[
A'(e_{j}):=
\begin{cases}
e_{j} & j\not\in S\\
0 & j\in S
\end{cases}
\]
and $b'=0$.  Hence, we can go from $\Delta_{S,T}$ to $P_{|T\backslash S|}$ by an affine transformation given by the composition of the above affine transformations.
\end{proof}

\begin{cor}
Let $\mathcal{D}$ be a subdivision of $\Delta_{S,T}$ induced by tropical Plücker coordinates $\{x_{I},y_{J}\}$, then $\mathcal{D}$ and $\Delta_{S,T}$ share the same 1-skeleton iff $\{x_{I},y_{J}\}$ satisfy $Trop(\sum_{i\in T\backslash S}X_{S\cup\{i\}}Y_{T\backslash\{i\}})$.
\label{intface}
\end{cor}

\begin{proof}
By the previous Proposition \ref{trans} we can convert the polytope $\Delta_{S,T}$ into $P_{|T\backslash S|}$, which means the subdivision, $\mathcal{D}$ on $\Delta_{S,T}$ becomes a subdivision $\mathcal{D'}$ on $P_{|T\backslash S|}$.  In particular, vertices $e_{S\cup\{i\}}$ and $e_{T\backslash\{i\}}$ in $\Delta_{S,T}$ become vertices $e_{j}$ and $-e_{j}$ in $P_{|T\backslash S|}$ for some $j$.  According to Proposition \ref{min}, the only way to introduce a new edge is for the weights $w_{i}$ of the subdivision to achieve the minimum exactly once.  The weights are given by the values on opposite vertices $e_{j}$ and $-e_{j}$, which in this case are given by $x_{S\cup\{i\}}$ and $y_{T\backslash\{i\}}$.  In particular, the weights are $w_{j}=x_{S\cup\{i\}}+y_{T\backslash\{i\}}$, a term of $Trop(\sum_{i\in T\backslash S}X_{S\cup\{i\}}Y_{T\backslash\{i\}})$.  By Proposition \ref{min}, the polytope has no new edges iff the $w_{i}$ achieve their minimum at least twice, which occurs iff $\{x_{I},y_{J}\}$ satisfy $Trop(\sum_{i\in T\backslash S}X_{S\cup\{i\}}Y_{T\backslash\{i\}})$.
\end{proof}

We are finally now in position to prove the following:

\begin{lem}
\label{edgeterm}
Given $I$ and $J$ multi-indices with $|I|=p$ and $|J|=q$, the following are equivalent:\\
(1) $e_{I}e_{J}$ is an edge of $\Delta$\\
(2) the term $X_{I}Y_{J}$ appears in no tropical incidence relation\\
(3) $I\subset J$
\end{lem}

\begin{proof}:\\
$((1)\Longrightarrow(2))$: We show the contrapositive: Suppose $X_{I}Y_{J}$ is a term in some incidence relation, call it $R_{S,T}$, then $I=S\cup\{i\}$ and $J=T\backslash\{i\}$.  By Proposition \ref{trans}, the vertices $e_{I},e_{J}$ are opposite each other, since they appear as $e_{r},-e_{r}$ in $P_{|T\backslash S|}$.  Since $e_{I}$ and $e_{J}$ are opposite each other $e_{I}e_{J}$ cannot be an edge of $\Delta_{S,T}$.  By Proposition \ref{face}, we know that $\Delta_{S,T}$ is a face of $\Delta$, hence $e_{I}e_{J}$ cannot be an edge of $\Delta$.\\
$((2)\Longrightarrow(3))$: Assume $I\not\subset J$ and let $r\in I\backslash J$.  Take $S=I\backslash\{r\}$, $T=J\cup\{r\}$.  Then $X_{I}Y_{J}$ is a term of the relation coming from S and T.\\
$((3)\Longrightarrow(1))$: Define the linear form $L(x)=\sum_{i\in I}x_{i}-\sum_{j\notin J}x_{j}$, then $L(x)$ gives a hyperplane in $\mathbb{R}^{n}$.  Note that $L(e_{S})\leq p$ for any multi-index set $S$.  In particular, $L(e_{S})=p$ only if $I\subset S\subset J$.  Hence $e_{I}e_{J}$ lies on the plane $L(x)=p$ while all other vertices of $\Delta$ are in the half-space $L(x)< p$. 
\end{proof}

\begin{thm}
Let $\{x_{I}\}$ with $(|I| = p)$, $\{y_{J}\}$ with $(|J| = q)$ be given.  The following are equivalent:\\
(1) The collection $\{x_{I}, y_{J}\}$ satisfies the tropical Plücker relations and tropical incidence relations.\\
(2) The subdivision $\mathcal{D}$ induced by the collection $\{x_{I}, y_{J}\}$ and $\Delta$ share the same 1-skeleton.
\label{main}
\end{thm}

\begin{proof}
Let us suppose that statement (1) holds, so that $\mathcal{D}$ is induced by points satisfying the Plücker relations and tropical incidence relations.  Now assume that $\mathcal{D}$ induces an internal edge $e_{I}e_{J}$ that is not an edge of $\Delta$.  Statements (1) and (3) of Lemma \ref{edgeterm} tell us that $e_{I}e_{J}$ is an edge of $\Delta$ iff $I\subset J$, so in this case we must have $I\not\subset J$.  Now by statement (2) of Lemma \ref{edgeterm} the term $X_{I}Y_{J}$ then appears in a tropical incidence relation $R_{S,T}$, which implies that $e_{I}e_{J}$ is an edge of the polytope $\Delta_{S,T}$ corresponding to $R_{S,T}$.  However, this means that $e_{I}e_{J}$ is an internal edge of $\mathcal{D}$ of the polytope $\Delta_{S,T}$, which contradicts Corollary \ref{intface}.  Hence, $\mathcal{D}$ is induced by points satisfying the tropical Plücker relations and tropical incidence relations if and only if $\mathcal{D}$ does not introduce new edges.
\end{proof}

\section{A Relationship with Matroids}

We have so far been able to show a relationship between tropical Plücker coordinates of tropical incidence varieties and certain polytopes.  We would like to be able to make a connection to matroids.  In particular, it would be nice if we could generalize Proposition 2.2 of \cite{Speyer}.  The statement is
\begin{propn} (Proposition 2.2 of \cite{Speyer}) The following are equivalent:\\
(1). $p_{i_{1},\ldots,i_{d}}$ are tropical Pl\"{u}cker coordinates\\
(2). The one skeleta of $\mathcal{D}$ and $\Delta(d,n)$ are the same.\\
(3). Every face of $\mathcal{D}$ is matroidal.
\end{propn}

\begin{defn}
Let $P$ be a polytope that is a subpolytope of $\Delta(d,n)$ for some $d$ and $n$.  $P$ is \textbf{matroidal} if its vertices, considered as subsets of ${[n]\choose d}$ are the bases of some matroid.
\end{defn}

Theorem \ref{main} gives us an analogous statement for (1) and (2) in the case of tropical incidence varieties.  Ideally, we would like to find a third condition using matroids that is equivalent to the statements of Theorem \ref{main}.  In particular, since $\Delta=Conv(\Delta(p,n),\Delta(q,n))$ we would like to consider polytopes of the form $Conv(F_{p},F_{q})$ where $F_{p}$ and $F_{q}$ are matroidal subpolytopes of $\Delta(p,n)$ and $\Delta(q,n)$ respectively.  What we then need is a condition \emph{between} two matroids.  Let $T_{p}$ and $T_{q}$ be the matroids corresponding to $F_{p}$ and $F_{q}$ respectively.  Since the polytope $\Delta$ arises in connection with conditions for when two tropical linear spaces are contained in each other, we would like $Conv(F_{p},F_{q})$ to do the same.  In particular, we want the bases of the matroids $T_{p}$ and $T_{q}$ to correspond to bases for linear spaces $V$ and $W$ where $W\subset V$, so a natural choice to consider is that of the matroid quotient:

There are many equivalent (cryptomorphic) ways to define a matroid quotient, a common one being that given matroids $M$ and $M'$ on the same underlying set $E$, we have $M'$ is a quotient of $M$ if the identity map on $E$ is a matroid strong map.  Since matroidal polytopes are defined using matroid bases, we need to define the matroid quotient in terms of bases as well.  The following is from Proposition 7.4.7 of \cite{Matroid}.  Given a matroid $M$, denote the \textbf{set of bases of $M$} by $\B(M)$.

\begin{defn} (Proposition 7.4.7, 1 in \cite{Matroid})
\label{quot}
$M'(E)$ is a \textbf{quotient} of $M(E)$, or $M'$ and $M$ are \textbf{concordant}, if and only if $\forall B\in\B(M)$ and $i\not\in B$, then $\exists B'\in\B(M')$ with $B'\subset B$ such that 
\begin{equation}
\{j;\;B\cup{i}\sub{j}\in\B(M)\}\supset\{j;\;B'\cup{i}\sub{j}\in\B(M')\}.
\label{conc}
\end{equation}
In other words, we have the following containment diagram:\\
\\
$\begin{CD}
B @>\cup{i}\sub{j}>> B\cup{i}\sub{j}\\
\bigcup @| \bigcup @|\\
B' @>\cup{i}\sub{j}>> B'\cup{i}\sub{j}
\end{CD}$\\
\\
\end{defn}
Note the following for vector matroids:

\begin{rem} (Proposition 7.4.8, 2 in \cite{Matroid})\\
Vector matroid quotients correspond to taking a quotient space: $M$ is represented by a subset $V'$ of vectors in a vector space $V$, and $M'$ is represented by vectors $\{v+W;\;v\in V'\}$ in the quotient space $V/W$ for some subspace $W$.  In particular, every vectorial nullity-$k$ quotient $M\mapsto M'$ can be realized by finding a matrix representation for $M$ that becomes a matrix representation for $M'$ on deletion of its first $k$ rows.
\end{rem}

This proposition suggests that matroid quotient is the relation we are looking for, as at least for vector matroids taking a quotient does in fact correspond to taking a quotient of vector spaces. Going on this expectation, we define the polytope that we would like to study:

\begin{defn}
Given matroidal polytopes $F$ and $G$ with corresponding matroids $T_{F}$ and $T_{G}$, we say that $Conv(F,G)$ is a \textbf{concordant polytope} if the matroids $T_{F}$ and $T_{G}$ are concordant.
\end{defn}

We can now make a third statement with matroids that we would like to be equivalent to the first two from Theorem \ref{main} giving:\\
\\
\textbf{Question}: Let $\mathcal{D}$ be a subdivision of $\Delta(p,q;n)$ arising from a point on $FD(p,q;n)$, then is the convex hull of every cell of $\mathcal{D}$ a concordant polytope, and does the converse hold?\\
\\
This question is supported in the affirmative by the following:\\
\\
\textbf{Observation}: If $V\subset W\subset K^{n}$ ($K$ the field of Puiseaux series in the variable $t$) are vector spaces (dim $V$=$p$ and dim $W$=$q$) and $X_{I}$ and $Y_{J}$ are tropical Pl\"{u}cker coordinates of $Trop(V)$ and $Trop(W)$ respectively, then every cell of the subdivision $\mathcal{D}$ induced by $X_{I}, Y_{J}$ is a concordant polytope.

\begin{proof}
Let $Q$ be a cell of the subdivision $\mathcal{D}$, then there exists a linear functional function $L$ satisfying for multi-indices $I$, $|I|=p$
\begin{align*}
L(e_{I})=X_{I},\: for\;e_{I}\in Q\\
L(e_{I})>X_{I},\: for\;e_{I}\notin Q.
\end{align*}
Similarly for multi-indices $J$, $|J|=q$ we have
\begin{align*}
L(e_{J})=Y_{J},\: for\;e_{J}\in Q\\
L(e_{J})>Y_{J},\: for\;e_{J}\notin Q.
\end{align*}
Now, letting $a_{R}:=L(e_{R})$ where $R$ is a multi-index of size $p$ or $q$ define
\begin{align*}
T:&K^{n}\rightarrow K^{n}\\
&e_{R}\mapsto t^{-a_{R}}e_{R}.
\end{align*}
This lets us replace $V$ and $W$ with the vector spaces $T(V)$ and $T(W)$, so that the $I$th tropical Pl\"{u}cker coordinate of $T(V)$ is now $X_{I}-L(e_{I})$, which is $0$ for $e_{I}\in Q$ and positive for $e_{I}\notin Q$ and similarly for coordinates of $W$.  Now let $\overline{V}$ and $\overline{W}$ be the reductions modulo $t$ of $T(V)$ and $T(W)$ respectively.  This translates to the $I$th Plücker coordinate of $\overline{V}$ being \textit{nonzero} if and only if $e_{I}\in Q$, and similarly the $J$th Plücker coordinate of $\overline{W}$ being nonzero is equivalent to $e_{J}\in Q$.\\
\\
Now note that the $I$th Plücker coordinate of $\overline{V}$ is nonzero is equivalent to $\{e_{k};\;k\in[n]\backslash I\}$ being a basis for $K^{n}/\overline{V}$, and similarly for the $J$th Plücker coordinate of $\overline{W}$ and $\{e_{k};\;k\in[n]\backslash J\}$.  Hence the matroid attached to $K^{n}/\overline{V}$ has bases $\{[n]\backslash I;\;e_{I}\in Q,\;|I|=p\}$, and similarly the matroid attached to $K^{n}/\overline{W}$ has bases $\{[n]\backslash J;\;e_{J}\in Q,\;|J|=q\}$.  We have $\overline{V}\subset\overline{W}\subset K^{n}$ so the matroid with bases $\{[n]\backslash J;\;e_{J}\in Q,\;|J|=q\}$ is a quotient of the matroid with bases $\{[n]\backslash I;\;e_{I}\in Q,\;|I|=p\}$.  The result follows by taking complements of the bases in $[n]$.
\end{proof}

Recall that the characterization of tropical linear spaces in terms of matroids is simply that every face of the induced subdivision of the hypersimplex is matroidal. This equivalence is obtained by looking at the edges that have to arise in the subdivision, and observing that the vertices of such an edge are related according to the exchange property of bases in matroids. However, we cannot prove the analogous statement, as posed in the question, as directly as in the case of matroids:  The following example shows that you can have cells of subdivisions that sit in the 1-skeleton of $\Delta$, but are \textit{not} concordant.  In particular it is possible for the polytope $Conv(F_{p},F_{q})$ to sit in the 1-skeleton of $\Delta$ while $T_{p}$ and $T_{q}$ are not concordant:

\begin{ex}
Let $F_{2}$ and $F_{3}$ be faces of $\Delta=\Delta(2,3;4)$ with vertices $\{e_{12},e_{13},e_{24},e_{34}\}$ and $\{e_{123},e_{124},e_{134}\}$ respectively.  Then the polytope $Conv(F_{2},F_{3})$ sits in the 1-skeleton of $\Delta$, but the matroids $T_{2}$ and $T_{3}$ with bases $\{12,13,24,34\}$ and $\{123,124,134\}$ respectively are not concordant:

\begin{figure}[h]

\tdplotsetmaincoords{60}{120}

\begin{tikzpicture}[>=stealth,tdplot_main_coords,scale=4]
  \coordinate[label=below right:{$\color{blue}12$}] (12) at (-0.707,0.707,0);
  \coordinate[label=below right:{$\color{blue}13$}] (13) at (0.707,0.707,0);
  \coordinate[label=above left:{$\color{blue}24$}] (24) at (-0.707,-0.707,0);
  \coordinate[label=above left:{$\color{blue}34$}] (34) at (0.707,-0.707,0);
  \coordinate[label=above right:{$\color{red}123$}] (123) at (0,1.225,0.707);
  \coordinate[label=above left:{$\color{red}124$}] (124) at (-0.707,0,0.707);
  \coordinate[label=below right:{$\color{red}134$}] (134) at (0.707,0,0.707);

  \draw[blue] (12) -- (13) -- (34) -- (24);
  \draw[blue,dashed] (24) -- (12);
  \draw[red] (123) -- (124) -- (134) -- cycle;
  \draw (123) -- (12);
  \draw (123) -- (13);
  \draw (134) -- (13);
  \draw (134) -- (34);
  \draw (124) -- (24);
  \draw[dashed] (124) -- (12);
\end{tikzpicture}

\caption{The polytope $Conv(F_{2},F_{3})$ with the vertices labeled with the corresponding bases.  The face $F_{2}$ is shown in blue, and $F_{3}$ in red.}

\end{figure}
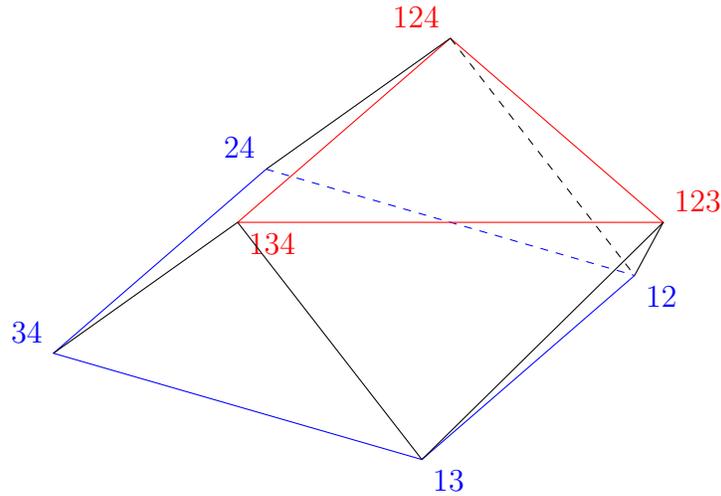

Note, however, that $Conv(F_{2},F_{3})$ cannot arise from a subdivision $\mathcal{D}$ of $\Delta(2,3;4)$ coming from a point on $FD(2,3;4)$: Any subdivision of $\Delta(2,3;4)$ induced by tropical Pl\"{u}cker coordinates containing $Conv(F_{2},F_{3})$ must introduce edges not in the 1-skeleton of $\Delta$.  If we attempt to do so by taking tropical Pl\"{u}cker coordinates $x_{12}=x_{13}=x_{24}=x_{34}=0$ and $y_{123}=y_{124}=y_{134}=0$ and the rest positive, not all the incidence relations are satisfied.  They can be if we impose either $x_{23}=0$ or $y_{234}=0$; then we are including either $23$ in $T_{2}$ or $234$ in $T_{3}$, but in either case the matroids become concordant.
\end{ex}

So we are left with the following:\\
\\
\textbf{Possibility}:
\label{3}
Let $\{x_{I}\}$ with $(|I| = p)$, $\{y_{J}\}$ with $(|J| = q)$ satisfy the tropical Plücker relations and tropical incidence relations.  Let $\mathcal{D}$ be the subdivision induced by $\{x_{I},y_{J}\}$, and let $F$ be a cell of $\mathcal{D}$.  The following are equivalent:\\
(1) $F$ has no internal edges.\\
(2) $F$ is concordant.\\
\\
While we have some evidence and partial results that suggest the above possibility may be true, a complete proof remains elusive.

\end{document}